\documentclass[11pt]{article}%{smfart}%
\usepackage{amssymb}\usepackage{amsmath}
\usepackage{amsthm}
\def\ST{{\mathsf T}}
\def\SP{{\mathsf P}}
\def\SR{{\mathsf R}}
\def\SX{{\mathsf X}}

\newtheorem{proposition}{\bfseries Proposition}
\newtheorem{lemma}{\bfseries Lemma}
\begin{document}  
 
\begin{center}
{\bf\large\sc   
Two relations for the antisymmetrizer in the Hecke algebra 
} 

\vspace*{2mm}
{ Andrei Bytsko   }
\vspace*{2mm}

\end{center}
\vspace*{2mm}

\begin{abstract} 
We prove two relations for the antisymmetrizer in the Hecke algebra and 
derive certain restrictions imposed by these relations on unitary representations 
of the Hecke algebra on tensor powers of the space ${\mathbb C}^n$.  
\end{abstract}
  
% \footnotetext 
% {
% Key words:  
% Hecke algebra, antisymmetrizer, unitary tensor spaces representation \\
% AMS Subject Classifications: 16S99, 15A24, 81R05 \\
% The work was supported in part by the program NCCR SwissMAP of
% the Swiss National Science Foundation. 
% }

\section{Introduction}

The Hecke algebra $H_N(q)$ is a unital associative algebra over ${\mathbb C}(q)$ with  
the generators $\SR_1, \ldots, \SR_{N-1}$ and relations
\begin{align}
\label{defH1} 
{}& \SR_k^2 = {\mathsf 1} + (q - q^{-1}) \, \SR_k  , \qquad k =1,\ldots, N-1 , \\[1mm]
\label{defH2} 
{}&   \SR_k \SR_m \SR_k = \SR_m \SR_k \SR_m , \qquad  |k-m| =1 ,\\[1mm] 
\label{defH3} 
{}&  \SR_k \SR_m =  \SR_m \SR_k ,  \qquad  |k-m| \geq 2 . 
\end{align}
The Hecke algebra can be regarded as a $q$-deformation of the group algebra of the symmetric 
group $S_N$. Relation (\ref{defH2}) coincides with the Yang-Baxter equation in the 
braid group form. For this reason, representations  of the Hecke algebra play important 
role in the theory of quantum groups \cite{Gur1,Jim1} and in the theory of  quantum 
integrable systems. 

Recall that a $q$-deformation of a number $k \in \mathbb Z$ is defined by  
\begin{align}\label{qnum} 
{}& [k] = \frac{q^k - q^{-k}}{ q - q^{-1} } .
\end{align}
We will need the following easily verifiable identities  
\begin{align}\label{qnid1} 
{}&  [k-1] + [k+1] = [2] [k], \qquad [k-1]  [k+1] +1 = [k]^2 .
\end{align}

The $q$-antisymmetrizer $\SP_N \in H_N(q)$ is defined  recursively by 
the relations (where by 
 $\SP_N \in H_{N+1}(q)$ is meant the element $\SP_N \otimes {id}$)   
\begin{equation}\label{PindH} 
 \SP_1 = {\mathsf 1} , \qquad  
 \SP_{N+1} = \frac{1}{[N+1]} \, \SP_N 
 \bigl( q^{N} -  [N] \, \SR_N \bigr)  \, \SP_N  . 
\end{equation} 
$\SP_N$ is an idempotent   
\begin{equation}\label{Psq} 
  \SP_{N}^2 = \SP_{N} , 
\end{equation} 
and it satisfies the following relations 
\begin{align}
\label{Qas1}
{}& ( \SR_k + q^{-1} \cdot {\mathsf 1}) \, \SP_N = \SP_N \, ( \SR_k + q^{-1} \cdot {\mathsf 1} ) = 0 , 
 \qquad     k= 1,\ldots, N-1 .
\end{align}

Let $\SP'_N$ denote the image of the element  $\SP_N \in H_{N+1}(q) $ 
under the shift of generators 
\begin{align}
\label{Rsdvig}
{}&  \SR'_k = \SR_{k+1} , \qquad k= 1,\ldots, N-1 . 
\end{align}

In the present work, we will establish the relation  
\begin{align}\label{PPP1} 
 {}&   (\SP_{N} - \SP_{N}' )^3   =  \frac{[N-1] [N+1]}{[N]^2} \, (\SP_{N} - \SP_{N}' )  .
\end{align}
Besides, we will show that the generators $\ST_ k = q^{-1} \cdot {\mathsf 1} + \SR_k$ 
satisfy the relation 
\begin{eqnarray}
\nonumber 
{}&  \displaystyle 
 \bigl( [2]^2 +1 \bigr) \, \ST_N \SP_N \ST_N = &  
   \frac{ [2] \bigl( [N+2] +2 [N] \bigr) }{[N]} \,  \ST_N \SP_{N-1} \\
\label{TPT1}
{}&  \displaystyle 
    &  - \frac{[N-1]}{[N]} \, \SP_{N-1} \ST_N \ST_{N-1} \ST_N \ST_{N-1}  \ST_N  \SP_{N-1}   
\end{eqnarray}
and we will consider some restrictions that relations (\ref{PPP1})  and  (\ref{TPT1}) 
impose on unitary representations of the Hecke algebra on spaces that are  tensor 
powers of ${\mathbb C}^n$.  In \cite{By3}, relation (\ref{PPP1}) was obtained for the 
Jones-Wenzl projector of the Temperley-Lieb algebra $TL_N(Q)$. 
We remark that, in the Temperley-Lieb algebra case, relation (\ref{TPT1}) takes the 
following simpler form  
\begin{eqnarray}
\label{TPTTL}
{}&  \displaystyle 
  \ST_N \SP_N \ST_N =  \frac{[N+1]}{[N]} \,   \ST_N \SP_{N-1} .
\end{eqnarray}

\section{Derivation of the relations for $\SP_N$}

Introduce new generators 
\begin{equation}\label{TRq} 
    \ST_ k = q^{-1} \cdot {\mathsf 1} + \SR_k    .
\end{equation}  
For these generators, relations (\ref{defH1})--(\ref{defH3}) are equivalent  to 
the following ones  
\begin{align}
\label{defHT1}
{}&  \ST_k ^2 = (q+q^{-1}) \ST_k ,   \qquad k =1,\ldots, N-1 , \\[1mm]
\label{defHT2} 
{}& \ST_k \ST_m \ST_k - \ST_m \ST_k \ST_m = \ST_k - \ST_m ,  \qquad  |k-m| =1 ,\\[1mm]
\label{defHT3} 
{}&  \ST_k \ST_m =  \ST_m \ST_k ,  \qquad  |k-m| \geq 2 ,
\end{align}
and formulae (\ref{PindH}), (\ref{Qas1}), and  (\ref{Rsdvig}) acquire the form 
\begin{align}
\label{PindHT} 
{}&  \SP_1 ={\mathsf 1} , \qquad 
\SP_{N+1} =  \SP_N - \rho_N \, 
 \SP_N \, \ST_N  \, \SP_N ,  \qquad   \rho_N = \frac{[N]}{[N+1]}  ,  \\[1mm] 
\label{Qas2} 
{}&  \ST_k \, \SP_N = \SP_N \, \ST_k = 0 , 
 \qquad     k= 1,\ldots, N-1 ,  \\[1mm] 
 \label{Tsdvig}
{}&  \ST'_k = \ST_{k+1} , \qquad k= 1,\ldots, N-1 . 
\end{align}

Let $\phi_N$ be an automorphism of the algebra $H_N(q)$ defined on the generators as follows 
\begin{equation}\label{phiT} 
    \phi_N \bigl( \ST_{k} \bigr) =  \ST_{N -k} , \qquad k =1,\ldots, N-1 .
\end{equation} 
{}From relations (\ref{Qas2}) we infer that   
$\ST_k  \, \phi_N \bigl( \SP_N ) = \phi_N \bigl( \SP_N )   \ST_k = 0$ 
for all $ k= 1,\ldots, N-1$. Relations (\ref{PindHT})  imply that
 $(\SP_N - {\mathsf 1})$ is a sum of monomials in 
$\ST_1$, \ldots, $\ST_{N-1}$. These properties allow us to check
that $\SP_N$ is invariant under the action of the automorphism $\phi_N$, 
\begin{equation}\label{phiinv} 
    \phi_N \bigl( \SP_N \bigr)  
 =   \phi_N \bigl( \SP_N \bigr) ({\mathsf 1} - \SP_N + \SP_N)
 = \phi_N \bigl( \SP_N \bigr)  \SP_N 
 = \bigl( \phi_N \bigl( \SP_N - {\mathsf 1}\bigr) +  {\mathsf 1} \bigr)  \SP_N
 = \SP_N .
\end{equation} 
Note that $\phi_{N+1} \bigl(\phi_{N} (\SX) \otimes id \bigr) = \SX'$ for any  
$\SX \in H_N(q)$.  In particular, taking into account the property (\ref{phiinv}), 
we see that 
$\phi_{N+1} \bigl( \SP_N \otimes id ) = \SP'_N$. 
Therefore, applying  the automorphism $\phi_{N+1}$ to the relation (\ref{PindHT}), 
we obtain  
\begin{align}\label{PPindHT} 
{}&  
\SP_{N+1} =  \SP'_N - \rho_N \, 
 \SP'_N \, \ST_1  \, \SP'_N  . 
\end{align} 

\begin{lemma} 
In the algebra $H_{N+2}(q)$, the following relations hold
\begin{eqnarray}
\label{delPP}
{}& 
 \SP_{N+1} - \SP'_{N+1}  = 
 \rho_N \, \SP'_N (  \ST_{N+1} - \ST_1 ) \SP'_N  , \\[2mm] 
\label{pttp1} 
{}&  
\rho_N \, \SP'_N \ST_1 \SP'_N \ST_1 \SP'_N  =  \SP'_N \ST_1 \SP'_N  , \\[2mm]
\label{pttp2} 
{}&
\rho_N \, \SP'_N \ST_{N+1} \SP'_N \ST_{N+1} \SP'_N  =  \SP'_N \ST_{N+1} \SP'_N  , \\[1mm]
\label{pttp3} 
{}& \displaystyle 
  \SP'_N \bigl(  \ST_1 \SP'_N \ST_{N+1} \SP'_N \ST_1 -
   \ST_{N+1} \SP'_N \ST_1 \SP'_N \ST_{N+1} \bigr) \SP'_N = 
   - \frac{[N+1]}{[N]^3} \bigl( \SP_{N+1} - \SP'_{N+1} \bigr) .
\end{eqnarray}
\end{lemma}

\begin{proof}

Applying the shift (\ref{Tsdvig}) to the relation (\ref{PindHT}), we obtain 
\begin{align}\label{PPPindHT} 
{}&  
\SP'_{N+1} =  \SP'_{N} - \rho_{N} \,  \SP'_{N} \, \ST_{N+1}  \, \SP'_{N}  . 
\end{align} 
Relation (\ref{delPP}) is the difference of relations (\ref{PPindHT}) and (\ref{PPPindHT}). 

Note that, as a consequence of relations (\ref{Qas2}), we have for $N \geq 2$ the equalities 
\begin{align}\label{PPpr} 
{}&  
% \SP_N  \SP_{N-1} = \SP_{N-1} \SP_N = \SP_N , \qquad
\SP'_N  \SP'_{N-1} = \SP'_{N-1} \SP'_N = \SP'_N , \quad
\ST_{N+1} \SP'_{N-1} =   \SP'_{N-1} \ST_{N+1} \quad  
\ST_{N} \SP'_{N} =   \SP'_{N} \ST_{N} = 0 .
\end{align} 
Let us prove relation (\ref{pttp2}). For $N=1$, it coincides with  
relation (\ref{defHT1}). For $N \geq 2$,  we have 
\begin{align} 
{}& \nonumber 
\SP'_N \ST_{N+1} \SP'_N \ST_{N+1} \SP'_N  \\
{}& \nonumber  
\stackrel{(\ref{PPPindHT})}{=}
\SP'_N (  \ST_{N+1} \SP'_{N-1}  \ST_{N+1}  
- \rho_{N-1} \, \ST_{N+1} \SP'_{N-1}   \ST_N  \, \SP'_{N-1} \ST_{N+1} ) \SP'_N \\
{}& \nonumber  
\stackrel{(\ref{PPpr})}{=}
\SP'_N (  \ST^2_{N+1}  
- \rho_{N-1} \, \ST_{N+1}   \ST_N    \ST_{N+1} ) \SP'_N \\
{}& \nonumber   
\stackrel{(\ref{defHT1}),(\ref{defHT2})}{=}
 \SP'_N (  [2] \ST_{N+1}  
- \rho_{N-1} \, ( \ST_N \ST_{N+1} \ST_N + \ST_{N+1} - \ST_N ) ) \SP'_N \\
{}& \label{ptprho1}  
\stackrel{(\ref{PPpr})}{=}  ([2]  - \rho_{N-1} ) \, \SP'_N \ST_{N+1} \SP'_N  
= \frac{1}{\rho_N} \, \SP'_N \ST_{N+1} \SP'_N .
\end{align}  
The last equality uses the identity 
\begin{align}\label{rhoNN} 
{}&  
[2]  - \rho_{N-1} =
 \frac{[2][N] -[N-1]}{[N]}  \stackrel{(\ref{qnid1})}{=}  \frac{[N+1]}{[N]}  
   = \frac{1}{\rho_N}  . 
\end{align}
Relation (\ref{pttp1}) can be obtained from relation (\ref{pttp2}) 
by applying the automorphism $\phi_{N+2}$ and taking into account that 
$\phi_{N+2}(\SP'_N) = \SP'_N$.  

Let us prove relation (\ref{pttp3}). For $N=1$, it coincides with 
relation (\ref{defHT2}) for $k=1$ and $m=2$ since we have  
$\SP_2 = {\mathsf 1} - \frac{1}{[2]} \ST_1$. In order to treat the case   
$N \geq 2$, we consider the difference of relations (\ref{PindHT}) and  (\ref{PPindHT}) 
(where $N$ is replaced with $(N+1))$:  
\begin{align} 
{}&  \nonumber 
\frac{1}{\rho_{N+1} } ( \SP_{N+1} - \SP'_{N+1} )    
=    \SP_{N+1}   \ST_{N+1}   \SP_{N+1} - \SP'_{N+1}   \ST_1   \SP'_{N+1}  \\
{}&  \nonumber  
\stackrel{(\ref{PPindHT}),(\ref{PPPindHT})}{=}
(\SP'_N - \rho_N \, \SP'_N   \ST_1  \SP'_N ) \ST_{N+1} (\SP'_N - \rho_N \, \SP'_N   \ST_1  \SP'_N ) \\[1mm] 
{}&  \nonumber  
\qquad\quad\ 
- (\SP'_N - \rho_N \, \SP'_N  \ST_{N+1}  \SP'_N ) \ST_1 (\SP'_N - \rho_N \, \SP'_N  \ST_{N+1}  \SP'_N ) \\[1mm]
{}&  \nonumber   
 = \SP'_N  ( \ST_{N+1} - \ST_1 )  \SP'_N  + \rho^2_N \, \SP'_N
  (\ST_1 \SP'_N \ST_{N+1} \SP'_N \ST_1 -
   \ST_{N+1} \SP'_N \ST_1 \SP'_N \ST_{N+1} )   \SP'_N \\ 
{}&  \nonumber   
\stackrel{(\ref{delPP})}{=} 
\frac{1}{\rho_N} ( \SP_{N+1} - \SP'_{N+1} )  + \rho^2_N \, 
 \SP'_N (\ST_1 \SP'_N \ST_{N+1} \SP'_N \ST_1 -
   \ST_{N+1} \SP'_N \ST_1 \SP'_N \ST_{N+1} )   \SP'_N . 
\end{align} 
The equality of the first and the last expressions is equivalent to relation (\ref{pttp3}) 
because we have the following identity  
\begin{align}\label{rhoNN2}
{}&   
 \frac{1}{\rho_N} - \frac{1}{\rho_{N+1}} 
 =  \frac{[N+1]}{[N]} - \frac{[N+2]}{[N+1]} 
 =  \frac{[N+1]^2 - [N] [N+2]}{[N] [N+1]} 
  \stackrel{(\ref{qnid1})}{=} \frac{1}{[N] [N+1]} .
\end{align}

\end{proof}
 
\begin{proposition}
In the algebra $H_{N+2}(q)$, relation (\ref{PPP1}) holds.  
\end{proposition}

\begin{proof}
\begin{align} 
{}&  \nonumber 
  \bigl( \SP_{N+1} - \SP'_{N+1}  \bigr)^2 
\stackrel{(\ref{delPP})}{=}  \rho_N^2 \, 
 \SP'_N (  \ST_{N+1} - \ST_1 ) \SP'_N   (  \ST_{N+1} - \ST_1 ) \SP'_N  \\[1mm] 
{}&  \nonumber 
 =  \rho_N^2 \,  \SP'_N (  \ST_{N+1} \SP'_N \ST_{N+1} + \ST_1 \SP'_N  \ST_1 
 - \ST_{N+1} \SP'_N \ST_1 - \ST_1 \SP'_N \ST_{N+1} )  \SP'_N  \\[1mm] 
{}&  \nonumber   
\stackrel{(\ref{pttp1}), (\ref{pttp2})}{=} 
\rho_N  \,  \SP'_N (  \ST_{N+1} + \ST_1  
 - \rho_N  \,  \ST_{N+1} \SP'_N \ST_1 - \rho_N  \,  \ST_1 \SP'_N \ST_{N+1} )  \SP'_N . 
\end{align}
With the help of this relation, we obtain 
\begin{align} 
{}&  \nonumber 
  \bigl( \SP_{N+1} - \SP'_{N+1}  \bigr)^3 \stackrel{(\ref{delPP})}{=}
  \rho_N \, \bigl( \SP_{N+1} - \SP'_{N+1}  \bigr)^2 \, 
  \SP'_N (  \ST_{N+1} - \ST_1 ) \SP'_N  \\[1mm] 
{}&  \nonumber 
= \rho_N^2 \,  \SP'_N (  \ST_{N+1} \SP'_N \ST_{N+1} - \ST_1 \SP'_N  \ST_1  
+ \ST_1 \SP'_N \ST_{N+1} - \ST_{N+1} \SP'_N  \ST_1  \\[1mm] 
{}&  \nonumber 
 \quad +  \rho_N \, \ST_{N+1} \SP'_N  \ST_1 \SP'_N  \ST_1 
 - \rho_N \, \ST_1 \SP'_N  \ST_{N+1} \SP'_N  \ST_{N+1}  \\[1mm] 
{}&  \nonumber 
 \quad  + \rho_N \, \ST_1 \SP'_N  \ST_{N+1} \SP'_N  \ST_1 
 -  \rho_N \, \ST_{N+1} \SP'_N  \ST_1 \SP'_N  \ST_{N+1}  )  \SP'_N  \\[1mm] 
{}&  \nonumber   
\stackrel{(\ref{pttp1}), (\ref{pttp2})}{=}  
 \rho_N  \,  \SP'_N (  \ST_{N+1} - \ST_1 
 + \rho_N^2 \, \ST_1 \SP'_N  \ST_{N+1} \SP'_N  \ST_1 
 -  \rho_N^2 \, \ST_{N+1} \SP'_N  \ST_1 \SP'_N  \ST_{N+1}  )  \SP'_N \\[1mm]   
{}&  \nonumber    
\stackrel{(\ref{delPP}),(\ref{pttp3})}{=}
\SP_{N+1} - \SP'_{N+1} - \frac{1}{[N+1]^2} \bigl( \SP_{N+1} - \SP'_{N+1} \bigr)
 \stackrel{(\ref{qnid1})}{=} 
 \frac{[N] [N+2]}{[N+1]^2} \bigl( \SP_{N+1} - \SP'_{N+1} \bigr) . 
\end{align}
\end{proof}
 
Let us remark that all the relations of Lemma 1 hold as well for the 
Temperley-Lieb algebra $TL_N(Q)$ with the parameter $Q=[2]$, i.e. 
in the case when the generators $\ST_k$ satisfy not the relation (\ref{defHT2})
but relation   
$\ST_k \ST_m \ST_k   = \ST_k $ for $|k-m| =1$. 
Therefore, Proposition~1 holds as well for the algebra $TL_N(Q)$.  
However, for the generators of the algebra $TL_N(Q)$ there are 
simpler relations  (cf. equations (28)--(31) in \cite{By3}), which 
allows to give a simpler derivation of relation (\ref{PPP1}) 
for the Jones-Wenzl  projector. 

\begin{proposition}\label{relTPT}
In the algebra $H_{N+1}(q)$, $N \geq 2$, relation (\ref{TPT1}) holds.
\end{proposition}

\begin{proof} 
\begin{align}
{}&  \nonumber 
  \ST_N \SP_N \ST_N  + \SP_N 
\stackrel{(\ref{PindHT}) }{=}  
 \ST_N ( \SP_{N-1} - \rho_{N-1} \SP_{N-1} \ST_{N-1} \SP_{N-1} ) \ST_N + \SP_N \\
{}& \nonumber 
 \stackrel{(\ref{PPpr}),(\ref{defHT1}) }{=}  
  [2] \ST_N \SP_{N-1}  - \rho_{N-1} \SP_{N-1} \ST_N \ST_{N-1} \ST_N \SP_{N-1} + \SP_N \\ 
{}& \nonumber 
 \stackrel{(\ref{defHT2}) }{=}   
 [2] \ST_N \SP_{N-1}  - \rho_{N-1} \SP_{N-1} \ST_{N-1} \ST_N \ST_{N-1} \SP_{N-1}
 - \rho_{N-1} \SP_{N-1} \ST_N \SP_{N-1} \\[1mm]   
{}& \nonumber 
 \qquad  + \rho_{N-1} \SP_{N-1} \ST_{N-1} \SP_{N-1}   + \SP_N   \\
{}& \label{TPTHe} 
 \stackrel{(\ref{PPpr}),(\ref{PindHT}) }{=}  
  ( [2] - \rho_{N-1} ) \ST_N \SP_{N-1} + \SP_{N-1}
   - \rho_{N-1} \SP_{N-1} \ST_{N-1} \ST_N \ST_{N-1} \SP_{N-1} .
\end{align}
Multiplying the derived equality by $\ST_N$ from both sides and using 
relation (\ref{defHT1}) and (\ref{PPpr}), we obtain  
\begin{eqnarray}
{}&   \nonumber 
 \bigl( [2]^2 +1 \bigr) \, \ST_N \SP_N \ST_N = & 
   \bigl( [2]^2 ( [2] - \rho_{N-1} ) + [2]  \bigr) \ST_N \SP_{N-1}    \\[1mm] 
{}& \label{TPT2}
   & - \rho_{N-1} \SP_{N-1} \ST_N \ST_{N-1} \ST_N \ST_{N-1} \ST_N \SP_{N-1} .
\end{eqnarray} 
Relations (\ref{TPT1}) and (\ref{TPT2}) are equivalent because the following 
identity holds 
\begin{align}\label{rhoNN4}
{}&   
 [2]  ( [2] - \rho_{N-1} ) + 1 
 \stackrel{(\ref{rhoNN}) }{=}  \frac{[2]}{\rho_N} + 1  = 
 \frac{[2] [N+1] +[N]}{[N]} 
  \stackrel{(\ref{qnid1})}{=}  \frac{ [N+2] + 2[N]}{[N]} . 
\end{align}
\end{proof} 

\section{On unitary representations on $\bigl({\mathbb C}^n \bigr)^{\otimes N}$ }

Let us regard the algebra $H_N(q)$ as a complex algebra with a parameter  
$q \in {\mathbb C}$, $|q|=1$ and an involution $*$ such that  
\begin{align}\label{invHR}
{}&   
  q^* = q^{-1}, \qquad \SR_k^* = \SR_k^{-1}  
\end{align}
The unitarity of the generators $\SR_k$ is equivalent to the hermiticity of the 
generators   $\ST_k$,
\begin{align}\label{invHT}
{}&   
   \ST_k^* = (q^{-1} \cdot {\mathsf 1} + \SR_k)^*
   = q \cdot {\mathsf 1} + \SR_k^{-1}    \stackrel{(\ref{defH1})}{=}
  q \cdot {\mathsf 1} +  \SR_k - (q - q^{-1})  \cdot {\mathsf 1} = \ST_k . 
\end{align}

Let $I \in \mathrm {Mat}(n,{\mathbb C})$ denote the identity matrix 
and $T \in \mathrm {Mat}(n^2,{\mathbb C})$ be a Hermitian matrix 
satisfying the relations   
\begin{align} 
{}& \label{TTTm} 
	 T^2 = (q+q^{-1}) \, T  , \qquad 
 T_{1} \, T_{2} \, T_{1} - T_{2} \, T_{1} \, T_{2}
  = T_{1}  - T_{2}  ,     
\end{align}
where $T_{1} \equiv T \,{\otimes}\, I$,  
$T_{2} \equiv I \,{\otimes}\, T$, and $\otimes$ stands for the Kronecker product.
Then a homomorphism   
$\tau: H_N(q) \to \mathrm {Mat}(n^N,{\mathbb C})$ such that  
\begin{align}\label{tauH}
 \tau(\ST_k) = T_{k} \equiv
 I^{\otimes (k-1)} \otimes T \otimes I^{\otimes (N-k-1)} 
\end{align}
is a $*$-representation of the algebra $H_N(q)$ on the tensor product space 
$\bigl({\mathbb C}^n \bigr)^{\otimes N}$.  In this representation,  
$\tau(\SR_k) = R_{k}$ are unitary matrices and they provide solutions to 
the Yang-Baxter equation   
\begin{align} \label{YB1}
{}&  R_{k} R_{k+1} R_{k} =  
  R_{k+1} R_{k}R_{k+1} .
\end{align}
Note that the matrices $R_{k}$ are unitary and  involutory if $q=1$.
A classification of solutions of the Yang-Baxter equation (\ref{YB1}) for
this case was obtained in~\cite{LPW}.  

Since we have $(q+q^{-1}) \in \mathbb R$ for $|q|=1$ and relations   
 (\ref{TTTm}) are invariant under the substitution $T_k \to - T_k$, $q \to -q$,  
it suffices to consider only the case   $(q+q^{-1}) >0$. 
(If $(q+q^{-1})=0$, then the only Hermitian matrix such that $T^2=0$ is 
$T=0$.)
 
\begin{proposition}\label{NgamPN}
Let $N \in \mathbb N$, $N \geq 2$ and $q = e^{i \gamma}$, where  
$ \frac{\pi}{N+1} \leq \gamma < \frac{\pi}{N}$. 
In this case, if a Hermitian matrix $T \neq 0$ satisfies relations   
 (\ref{TTTm}) and  $\tau$ is the corresponding representation (\ref{tauH}), then  
$\tau(\SP_N) = 0$.
\end{proposition}

\begin{proof}
Set $P_N = \tau(\SP_N)$. 
Since $ 0 < \gamma < \frac{\pi}{N}$, then    
  $\rho_n = \frac{[n]}{[n+1]}=\frac{\sin(n \gamma)}{\sin((n+1) \gamma)}$ 
are well defined (and positive) for all $n=1,\ldots, N-1$. 
Therefore, the projectors $P_k$ defined recursively by the formula 
(\ref{PindHT}) exist for all $k=1,\ldots, N$.  
The same formula (\ref{PindHT}) implies also that all the projectors  $P_k$ 
are Hermitian and hence that $(P_{N} - P_{N}' )^4$ and 
$(P_{N} - P_{N}' )^2$ are positive semi-definite matrices. 
But relation (\ref{PPP1}) implies the equality   
\begin{equation}\label{PPP2} 
    (P_{N} - P_{N}' )^4  =  \frac{[N-1] [N+1]}{[N]^2} (P_{N} - P_{N}' )^2 , 
\end{equation}
where $\frac{[N-1] [N+1]}{[N]^2} \leq 0$ because $(N-1) \gamma <\pi$ and  
$ \pi \leq (N+1) \gamma  < 2\pi$. That is, the right hand side of (\ref{PPP2}) 
is a negative semi-definite matrix. Thus, equality (\ref{PPP2}) can hold only 
if  $P_{N} = P_{N}' $.
But then we have $0= T_1 P_N =  T_1 P''_N  =  T \otimes P_N$. Whence $P_N =0$
because $T \neq 0$. 
\end{proof}

Let us remark that the statement proved above is similar to the statement 
about  $C^*$-representations of the Hecke algebra on a Hilbert space 
obtained in  \cite{We1}. In the case that we consider here, the translational invariance 
of representation (\ref{tauH}) leads to additional restrictions. In particular,
it was shown in \cite{By3} that unitary representations of the Temperley-Lieb algebra 
on the tensor product space $\bigl({\mathbb C}^n \bigr)^{\otimes N}$ 
exist only if $|q+q^{-1}|=1,\sqrt{2},\sqrt{3},2$.

\begin{proposition}\label{NNgamP}
Let $N \in \mathbb N$, $N \geq 3$ and $q = e^{i \gamma}$, where  
$ \frac{\pi}{N+1} < \gamma < \frac{\pi}{N}$. Suppose, in addition, that
\begin{eqnarray}
\label{NN20}
[N+2] +2 [N] < 0 .
\end{eqnarray}
In this case, if a Hermitian matrix $T \neq 0$ satisfies relations   
 (\ref{TTTm}) and  $\tau$ is the corresponding representation (\ref{tauH}), then 
$\tau(\SP_{N-1}) = 0$.
\end{proposition}

\begin{proof} 
By Proposition \ref{NgamPN}, we have $P_N =0$. Since, furthermore, 
we have $[N] >0$, then relation (\ref{TPT1}) acquires the form   
\begin{eqnarray}
\label{TPTgams}
{}&  \displaystyle 
    [2] \bigl( [N+2] +2 [N] \bigr)   \,  T_N P_{N-1}  = 
  [N-1]  \, P_{N-1} T_N T_{N-1} T_N T_{N-1} T_N  P_{N-1} . 
\end{eqnarray} 
The right hand side of (\ref{TPTgams}) is a positive semi-definite matrix since 
$[N-1] >0$. The matrix  
$[2]  T_N P_{N-1} =(T_N P_{N-1}) (T_N P_{N-1})^* $ 
is also positive semi-definite because $T_N$ commutes with $P_{N-1}$. 
Therefore, if the condition (\ref{NN20}) is satisfied, then the left hand side of  
(\ref{TPTgams}) is a negative semi-definite matrix. Thus, relation (\ref{TPTgams})  
can hold only if  $T_N P_{N-1} =0$.
But $T_N P_{N-1} = P_{N-1} \otimes T$, where $T \neq 0$. Whence $ P_{N-1} =0$. 
\end{proof} 

Let us remark that we have $[N+2] +2 [N] = 1$ for $\gamma = \frac{\pi}{N+1}$ and  
$[N+2] +2 [N] = -[2] < 0$ for $\gamma = \frac{\pi}{N}$. Therefore, on  
every interval  $\bigl( \frac{\pi}{N+1} , \frac{\pi}{N} \bigr)$, $N \geq 3$, 
there exist values of $\gamma$ for which the condition  (\ref{NN20}) is satisfied.

\begin{proposition}
Let $q = e^{i \gamma}$, where  
$\gamma \in \bigl(\arccos (   \sqrt{ \frac{1}{8} (1 + \sqrt{5})} , \frac{1}{2} \pi \bigr)$.
If a Hermitian matrix $T \neq 0$ satisfies relations (\ref{TTTm}),  
then $T = [2] \, (I \otimes I)$.
\end{proposition}

\begin{proof}
Consider the corresponding representation (\ref{tauH}). 
For $\gamma \in [\frac{1}{3} \pi, \frac{1}{2} \pi )$, we have $T = [2] \, (I \otimes I)$
because $P_2=0$ by Proposition~\ref{NgamPN}. 

For $N=3$, the condition (\ref{NN20}) is satisfied if 
$f(\gamma) \equiv \sin 5 \gamma  +2 \sin 3 \gamma <0$. The equation 
$f(\gamma)  =  (16 \cos ^4\gamma - 4 \cos^2 \gamma -1) \sin \gamma=0$
has on the interval $(\frac{1}{4} \pi, \frac{1}{3} \pi )$ a single root which 
corresponds to the value $\cos^2 \gamma_0 = \frac{1}{8}(1+\sqrt{5})$. 
Taking into account that $f(\frac{1}{3}\pi) < 0$, we conclude that,  for 
$\gamma \in \bigl( \gamma_0, \frac{1}{3}\pi)$,  the condition (\ref{NN20}) 
is satisfied. But then we have $P_2=0$ by Proposition~\ref{NNgamP} 
and, hence,  $T = [2] \, (I \otimes I)$. 
\end{proof}  

\vspace*{1mm}
\small{
{\bf Acknowledgements.} 
This work was supported in part by the program NCCR SwissMAP of
the Swiss National Science Foundation.  
}

 \vspace*{1.5mm}
{\sc \small
\noindent
Section of Mathematics, University of Geneva,  
C.P. 64, 1211 Gen\`eve 4, Switzerland  \\[1mm] 
Steklov Mathematical Institute, 
Fontanka 27, 191023, St. Petersburg, Russia  
}

\end{document}